\documentclass[11pt]{amsart}

\usepackage{amsmath}
\usepackage{amssymb}
\usepackage{graphicx}
\usepackage[
citestyle=numeric,
bibstyle=numeric
]{biblatex}
\usepackage{biblatex}

\newtheorem{thm}{Theorem}[section]
\newtheorem*{thm*}{Theorem}

\newtheorem{lem}[thm]{Lemma}
\newtheorem*{lem*}{Lemma}
\newtheorem*{cor*}{Corollary}
\theoremstyle{definition}
\newtheorem{defn}[thm]{Definition}
\newtheorem{cor}{Corollary}[section]
\newtheorem{ex}[thm]{Example}

\newtheorem{rem}[thm]{Remark}

\numberwithin{equation}{section}

\title{The Kaczmarz Algorithm in Hilbert $C^{*}$-modules}
\author{Daniel Alpay, Chad Berner, and Eric S. Weber}

\begin{document}

\maketitle

\begin{abstract}
The Kaczmarz algorithm in Hilbert spaces is a classical iterative method for stably recovering vectors from inner product data. In this paper, we extend the algorithm to the setting of Hilbert $C^*$-modules and establish analogues of its effectiveness in both finite-dimensional and stationary cases. Consequently, we demonstrate that continuous families of elements in a Hilbert space can be uniformly recovered using the Kaczmarz algorithm. Additionally, we develop a normalized Cauchy transform for continuous families of measures and use it to provide sufficient conditions under which standard frames in Hilbert $C(X)$-modules can be generated by the Kaczmarz algorithm and realized as orbits of bounded operators.
\end{abstract}

\section{Introduction}

In 1937, Stefan Kaczmarz introduced an iterative process for solving linear systems which is now known as the \textit{Kaczmarz algorithm} \cite{Kaczmarz1993Approximate}.   At its root, the algorithm's objective is to recover an unknown vector $x$ from a (finite or infinite) sequence of inner products $\{ \langle x, e_{k} \rangle \}$. The Kaczmarz algorithm can be extended to infinite dimensional Hilbert spaces \cite{Kwapien2001Kaczmarz,Haller2005Kaczmarz} as follows.  Given a sequence of unit vectors $\{e_n\}_{n=0}^{\infty}$ in a Hilbert space $\mathcal{H}$ and $x \in \mathcal{H}$, we define a sequence of approximations $\{x_n\}_{n=0}^{\infty}$ by
\begin{equation}\label{defn:ka}
\begin{aligned}
		x_0 &= \langle x, e_0 \rangle e_0, \\
		x_n &= x_{n-1} + \langle x- x_{n-1}, e_n \rangle e_n, \quad n \geq 1.
\end{aligned}
\end{equation}
The sequence $\{e_n\}_{n=0}^{\infty}$ is said to be \textit{effective} if $\| x_n - x \| \to 0$ for every $x \in \mathcal{H}$. Kaczmarz showed in \cite{Kaczmarz1993Approximate} that any periodic, linearly dense sequence of unit vectors $\{e_n\}_{n = 0}^{\infty}$ in a finite-dimensional Hilbert space is effective. 

We note that the definition of the Kaczmarz update in Equation \eqref{defn:ka} illustrates a crucial feature of the method: the successor approximation is the orthogonal projection of the current estimate onto the solution space of the current equation (i.e. the hyperplane of all $y$ in $H$ determined by $\langle y, e_{n} \rangle = \langle x, e_{n} \rangle$).  Consequently, the method strongly leverages the geometry of Hilbert spaces--indeed, in a heuristic sense, the algorithm converges to the solution because successive approximations are never further away than prior estimates.

Nonetheless, there is interest in extending the Kaczmarz algorithm beyond the Hilbert space context.  Kwapie\'n and Mycielski \cite{Kwapien2001Kaczmarz} consider the question of effectivity in the space $C(\mathbb{T})$.  They consider the problem of whether there exists a sequence $\{ t_{n} \} \subset \mathbb{T}$ such that an unknown function $f$ can be recovered from the data $\{ f(t_{n}) \}$ through successive approximations analogous to the Kaczmarz update.  They conceptualize this problem as \emph{online learning}, and prove that a Kaczmarz inspired algorithm can recover $f$ up to a finite horizon.  A dual Kaczmarz algorithm was introduced in \cite{Aboud2020Dual} for Hilbert spaces in which the estimates for $x$ from the data $\{ e_{n}(x) \}$ (here we use functional notation rather than inner-products) are constructed using a second sequence $\{ f_{n} \}$.  Consequently, an analysis of effectivity in Banach spaces is presented in \cite{Aboud2022Kaczmarz}; the question of effectivity depends on the choice of two sequences--$\{e_{n}\} \subset X$ .  Because projections in Banach spaces are not necessarily norm-decreasing, the convergence analysis is more subtle and unknown in general.  Kwapie\'n and Mycielski observe that by the Banach-Steinhaus theorem, uniform boundedness of the products of projections are necessary for a sequence to be effective.

In this paper, we introduce the Kaczmarz algorithm in the context of Hilbert $C^*$-modules.  In the classical case of a Hilbert space, the convergence analysis of the Kaczmarz algorithm utilizes the spectral radius of a product of projections to show effectivity \cite{Natterer1986Mathematics}.  Additionally, the Pythagorean theorem contributes to the analysis.  Our results in a Hilbert $C^*$-module $H$, we will instead appeal to the notion of Parseval frames in $H$ \cite{Frank2002Frames}.  While the Pythagorean theorem need not hold in $H$, we will leverage a norm-decreasing property of orthogonal decompositions, see Lemma \ref{normdeclem}.

The outline of this paper is as follows: In section \ref{S3}, we begin by developing a Kaczmarz algorithm in the setting of Hilbert $C^{*}$-modules, similarly to the Hilbert space setting. Then we establish necessary and sufficient conditions for effectivity of this algorithm using a Parseval frame-type condition, similarly to the Hilbert space setting. Furthermore, we demonstrate that a finite-dimensional analogue of the algorithm holds as expected in the Hilbert $C^{*}$-module setting.

In section \ref{S4}, we establish necessary and sufficient conditions for effectivity of the Kaczmarz algorithm when the underlying sequence is stationary. Then we show that in this setting when the $C^{*}$-algebra is commutative, the effectivity of the sequence only depends on the fibers of the inner product space. This allows us to conclude that a continuous family of elements of a Hilbert space can be recovered uniformly by an effective stationary sequence. The section concludes by constructing a continuous field of $L^{2}$ spaces and characterizing when exponential functions form an effective sequence in this context.

Finally, section \ref{S5} establishes the image of the inner product space arising from our continuous field of $L^{2}$ spaces from section \ref{S4}. This allows us to define a normalized Cauchy transform for a continuous family of measures. As a consequence, we show that standard Parseval frames in Hilbert $C(X)$-modules—whose inner product spaces align with those from our continuous field—arise from an effective Kaczmarz algorithm and can be realized as orbits of bounded operators.

Throughout this paper, for a ring $A$, we denote $A[[x]]$ as the ring of power series with coefficients in $A$. Additionally, for a Borel probability measure $\mu$ on $[0,1)$, we denote $$L^{2}(\mu)^{+}:=\overline{span\{e^{2\pi i nx} : n\in \mathbb{N}\}}\subseteq L^{2}(\mu).$$

\section{Preliminaries}
We begin by defining Hilbert $C^{*}$ modules. Note that we assume the $C^{*}$-algebras in this paper are always unital to avoid complications with $\mathbb{C}$-linearity and linearity of the $C^{*}$-algebra valued inner products.
\begin{defn}
Let $A$ be a unital $C^{*}$ algebra. A \textbf{pre-Hilbert $A$-module} is a complex vector space and an algebraic (left) $A$-module $H$, equipped with an inner product $\langle \cdot, \cdot\rangle: H^{2}\to A$ with the following properties:
\begin{enumerate}
    \item $\langle x, x\rangle \geq 0$ for all $x\in H$ with equality if and only if $x=0$.
    \item $\langle x, y\rangle =\langle y, x\rangle^{*}$ for all $x,y\in H$.
    \item $\langle ax+by,z\rangle=a\langle x,z\rangle +b\langle y, z\rangle $ for all $x,y,z\in H$ and $a,b\in A$.
\end{enumerate}
Now define $$\|x\|:=\sqrt{\| \langle x, x\rangle \|_{A}}$$ for any $x\in H$. If $H$ is complete with respect to this norm, we say $H$ is a \textbf{Hilbert A-module}.
\end{defn}

The reader should note that even though Hilbert $C^{*}$-modules are Banach spaces equipped with a type of inner product, lots of desirable properties that are available in Hilbert spaces may be unavailable in the Hilbert $C^{*}$-module setting. For example, there exists Hilbert $A$-modules $H$ with bounded $A$-linear maps from $H$ to $A$ that can not be represented as an inner product. See \cite{Frank2002Frames} for explicit examples. If a Hilbert $A$-modules $H$ has the property that all bounded $A$-linear maps from $H$ to $A$ are given by inner products, $H$ is called \textbf{self-dual}.

Furthermore, if $H,K$ are Hilbert $A$-modules, and $B: H\to K$ is a bounded $A$-linear map, there may not be a bounded $A$-linear map $B^{*}: K\to H$ such that
$$\langle Bx,y\rangle =\langle x, B^{*}y\rangle$$ for all $x\in H$ and $y\in K$. If such $B$ have a $B^{*}$ with this property, we say $B$ is \textbf{adjointable} and we call $B^{*}$ the \textbf{adjoint} of $B$.

One property of Hilbert spaces that does carry over nicely in some sense to the Hilbert $C^{*}$-module setting is the Cauchy–Schwarz inequality. Specifically, for any $x,y\in H$ where $H$ is a Hilbert $A$-module, we have
$$\|\langle x, y\rangle \|_{A}\leq \|x\| \|y\|.$$ For a proof, the reader can check out Landsmans's lecture notes \cite{Landsman1998Lecture}.

The proof of this result uses the following fact about $C^{*}$-algebras that we will use several times in our arguments throughout this paper:
\begin{lem}\label{normdeclem}
Suppose that $C,B\geq 0$ as elements of $A$ where $A$ is a unital $C^{*}$-algebra. Then 
$$\|C\|, \|B\|\leq \|C+B\|.$$
\end{lem}
\begin{proof}
Assuming that $C,B\in B(H')$ are positive semi-definite for some Hilbert space $H'$ by the Gelfand-Naimark theorem,
$$\|C\|=sup_{\|f\|=1}\langle Cf, f\rangle \leq sup_{\|f\|=1}\langle (C+B)f, f\rangle=\|C+B\|.$$

Similarly, $$\|B\|\leq \|C+B\|.$$
\end{proof}

Throughout this paper, our development of the Kaczmarz algorithm is closely tied with frames in the Hilbert $C^{*}$-module setting. For our purposes, we are specifically working with standard frames, which are defined by Frank and Larson in \cite{Frank2002Frames}:
\begin{defn}
Let $A$ be a unital $C^{*}$-algebra and $H$ be a Hilbert $A$-module. A sequence $\{x_{n}\}_{n=0}^{\infty}$ is called a \textbf{standard frame} if
there is an $a,b>0$ where
$$a\langle x, x\rangle \leq \sum_{n=0}^{\infty}\langle x, x_{n}\rangle \langle x_{n}, x\rangle \leq b\langle x, x\rangle$$
for all $x\in H$ where the sum in the middle of the equality converges in the norm of $A$.

If the inequality above is an equality, we say $\{x_{n}\}_{n=0}^{\infty}$ is a \textbf{standard Parseval frame}.
\end{defn} For more about frames in the setting of Hilbert $C^{*}$-modules, the reader can see \cite{Arambasic2017Frames}.

The following theorem in \cite{Frank2002Frames} establishes a stable recovery for standard frames via a type of frame operator, analogously to the Hilbert space setting:

\begin{thm*}[Frank and Larson]
Let $A$ be a unital $C^{*}$-algebra and $H$ be a Hilbert $A$-module. Suppose $\{f_{n}\}_{n=0}^{\infty}\subseteq H$ is a standard frame. Then there is a unique adjointable operator $S: H\to H$ such that
$$f=\sum_{n=0}^{\infty}\langle S(f), f_{n}\rangle f_{n}$$
for all $f\in H$.
\end{thm*}

Another property that can be lost when moving from the Hilbert space to the Hilbert $C^{*}$-module setting is always having an orthonormal basis.
We now define a reasonable notion of an orthonormal basis and a generating set for Hilbert $C^{*}$-modules: 

\begin{defn}
Let $A$ be a unital $C^{*}$-algebra and $H$ be a Hilbert $A$-module. We say $\{e_{n}\}_{n=0}^{\infty}\subseteq H$ \textbf{generates} $H$ if the $A-$linear combination of elements from $\{e_{n}\}_{n=0}^{\infty}$ is dense in $H$.
\end{defn}

\begin{defn}
Let $A$ be a unital $C^{*}$-algebra and $H$ be a Hilbert $A$-module. We say that $\{e_{n}\}_{n=0}^{\infty}\subseteq H$ is an \textbf{orthonormal basis} if for all $n, k$,
$$\langle e_{n},e_{k}\rangle=\delta_{nk},$$ and $\{e_{n}\}_{n=0}^{\infty}$ generates $H$.
\end{defn}

It is not hard to check that an orthonormal basis in a Hilbert $C^{*}$-module is a standard Parseval frame. However, not every Hilbert $C^{*}$-module possesses even standard frames. To see an example, the reader can check out Li's work in \cite{Li2010Hilbert}. For our purposes, we will only consider countably generated Hilbert $C^{*}$-modules, which always admit frames, see \cite{Frank2002Frames}. For more about countably generated Hilbert $C^{*}$-modules, the reader can see \cite{Raeburn2003Countably}.

We will also draw on theory from the Hardy space in our development of the Kaczmarz algorithm:
\begin{defn}
Recall the Hardy space, which can be defined as follows:
$$H^{2}(\mathbb{D})=\{w\to \sum_{n=0}^{\infty}c_{n}w^{n}: \sum_{n=0}^{\infty}|c_{n}|^{2}<\infty\}.$$
Furthermore, an \textbf{inner function} $b(w)$ is a bounded analytic function on $\mathbb{D}$ such that 
$$||b(w)f(w)||_{H^{2}(\mathbb{D})}=||f(w)||_{H^{2}(\mathbb{D})}$$ for any $f\in H^{2}(\mathbb{D})$.
\end{defn}

\section{A Kaczmarz algorithm for Hilbert $C^{*}$-modules}\label{S3}
We begin by establishing our Kaczmarz algorithm in the Hilbert $C^{*}$-module setting in a manner similar to that of Kwapien and Mycielski in the Hilbert space setting \cite{Kwapien2001Kaczmarz}. Note that we assume our sequences that we use to reconstruct vectors in the Hilbert $C^{*}$-module are not only unit vectors in the Banach space, but also unit vectors in the $C^{*}$-algebra valued inner product sense as well because without this assumption our theory becomes much more complicated.

\begin{defn}\label{kaczmarzdef}
Let $A$ be unital $C^{*}$ algebra and $H$ be a Hilbert $A$-module. Suppose that for $\{e_{n}\}_{n=0}^{\infty}\subseteq H$, $$\langle e_{n},e_{n}\rangle=I$$ for all $n\geq 0$.

For $x\in H$, let
$$x_{0}=\langle x,e_{0}\rangle e_{0} $$
$$x_{n}=x_{n-1}+\langle x-x_{n-1},e_{n}\rangle e_{n}, \ \forall n\geq 1.$$
If $x_{n}\to x$ in norm $\forall x\in H$, then we say \textbf{$\{e_{n}\}_{n=0}^{\infty}$ is effective}.   
\end{defn}

Our first goal is to establish a necessary and sufficient condition for $\{e_{n}\}_{n=0}^{\infty}\subseteq H$ to be effective using a secondary sequence and a Parseval frame condition. This condition is analogous to the Parseval frame condition established by Haller and Szwarc and Kwapien and Mycielski in the Hilbert space setting \cite{Haller2005Kaczmarz, Kwapien2001Kaczmarz}. To achieve this result, we prove two straightforward lemmas:

\begin{lem}\label{orthoglemma}
Let $A$ be a unital $C^{*}$-algebra and $H$ be a Hilbert $A$-module. Suppose that for $\{e_{n}\}_{n=0}^{\infty}\subseteq H$, $$\langle e_{n},e_{n}\rangle=I$$ for all $n\geq 0$. For any $x\in H$,
$$\langle x-x_{n},e_{n}\rangle=0$$ for all $n\geq 0$.
\end{lem}
\begin{proof}
Note that for any $x\in H$,
$$x-x_{n}=x-x_{n-1}-\langle x-x_{n-1}, e_{n}\rangle e_{n}, \ \forall n\geq 1.$$
Therefore,
$$\langle x-x_{n},e_{n}\rangle =\langle x-x_{n-1}, e_{n}\rangle -\langle x-x_{n-1},e_{n}\rangle \langle e_{n}, e_{n}\rangle =0, \ \forall n\geq 1.$$ Note the $n=0$ case follows from Definition \ref{kaczmarzdef}.
\end{proof}
Before we prove the other lemma, we define our secondary sequence:
\begin{defn}\label{auxseq}
Let $A$ be a unital $C^{*}$-algebra and $H$ be a Hilbert $A$-module. Suppose that for $\{e_{n}\}_{n=0}^{\infty}\subseteq H$, $$\langle e_{n},e_{n}\rangle=I$$ for all $n\geq 0$.
Define
$$g_{0}=e_{0}$$ and
$$g_{n}=e_{n}-\sum_{k=0}^{n-1}\langle e_{n},e_{k}\rangle g_{k}, \ \forall n\geq 1.$$
We call $\{g_{n}\}_{n=0}^{\infty}$ the \textbf{auxiliary sequence of} $\{e_{n}\}_{n=0}^{\infty}$.
\end{defn}

\begin{lem}\label{auxseqform}
Let $A$ be a unital $C^{*}$-algebra and $H$ be a Hilbert $A$-module. Suppose that for $\{e_{n}\}_{n=0}^{\infty}\subseteq H$, $$\langle e_{n},e_{n}\rangle=I$$ for all $n\geq 0$.
Then $\forall n\geq 0$,
$$x_{n}=\sum_{k=0}^{n}\langle x,g_{k}\rangle e_{k}.$$
\end{lem}
\begin{proof}
The case for $n=0$ is obvious.

Suppose the conclusion holds for all $0\leq k\leq n$. Then

$$x_{n+1}=\sum_{k=0}^{n}\langle x,g_{k}\rangle e_{k}+ \langle x-\sum_{k=0}^{n}\langle x,g_{k}\rangle e_{k}, e_{n+1}\rangle e_{n+1}$$
$$=\sum_{k=0}^{n}\langle x,g_{k}\rangle e_{k} + \langle x, e_{n+1}\rangle e_{n+1} -\langle x, \sum_{k=0}^{n}\langle e_{n+1},e_{k}\rangle g_{k}\rangle e_{n+1}$$
$$=\sum_{k=0}^{n+1}\langle x, g_{k}\rangle e_{k}.$$
\end{proof}

\begin{thm}\label{pframecond}
Let $A$ be a unital $C^{*}$-algebra and $H$ be a Hilbert $A$-module. Suppose that for $\{e_{n}\}_{n=0}^{\infty}\subseteq H$, $$\langle e_{n},e_{n}\rangle=I$$ for all $n\geq 0$.
Then $\{e_{n}\}_{n=0}^{\infty}$ is effective if and only if $\{g_{n}\}_{n=0}^{\infty}$ is a standard Parseval frame.
\end{thm}

\begin{proof}

Note for any $x\in H$ and $n\geq 1$ by Lemma \ref{orthoglemma} and Lemma \ref{auxseqform},
$$\langle x-x_{n-1},x-x_{n-1} \rangle = \langle x-x_{n}+\langle x, g_{n}\rangle e_{n},x-x_{n}+\langle x, g_{n}\rangle e_{n}\rangle$$
$$=\langle x-x_{n}, x-x_{n}\rangle +\langle x,g_{n}\rangle \langle g_{n},x\rangle .$$

It follows that
$$\langle x-x_{0}, x-x_{0} \rangle =\langle x-x_{n}, x-x_{n} \rangle + \sum_{k=1}^{n}\langle x, g_{k}\rangle \langle g_{k}, x\rangle=\langle x, x\rangle -\langle x, g_{0}\rangle \langle g_{0}, x\rangle .$$

Therefore, $\{e_{n}\}_{n=0}^{\infty}$ is effective if and only if

$$\sum_{k=0}^{\infty}\langle x, g_{k}\rangle \langle g_{k}, x\rangle =\langle x,x\rangle $$ in the norm of $A$ for all $x\in H$.
\end{proof}

Our next result is an analogue to Czaja and Tanis's result that says an effective Riesz basis must be an orthonormal basis \cite{Czaja2013Kaczmarz}. Note however, that our result does not require that our effective sequence is a frame:
\begin{thm}
Let $A$ be a unital $C^{*}$-algebra and $H$ be a Hilbert $A$-module. Suppose that for $\{e_{n}\}_{n=0}^{\infty}\subseteq H$, $$\langle e_{n},e_{n}\rangle=I$$ for all $n\geq 0$.
Suppose further that $$\sum_{n=0}^{\infty}c_{n}e_{n}=0 \ \text{for some} \ \{c_{n}\}_{n=0}^{\infty}\subseteq A\implies \ c_{n}=0, \  \forall n.$$
Then $\{e_{n}\}_{n=0}^{\infty}$ is effective if and only if $\{e_{n}\}_{n=0}^{\infty}$ is an orthonormal basis.
\end{thm}

\begin{proof}
Assume that $\{e_{n}\}_{n=0}^{\infty}$ is effective.
Then for all $k$, 
$$e_{k}=\sum_{n=0}^{\infty}\langle e_{k}, g_{n}\rangle e_{n}\implies \langle e_{k},g_{n}\rangle=0, \ \forall n\neq k \ \text{and} \ \langle e_{k},g_{k}\rangle=I.$$

We prove that for all $n$, $$g_{n}=e_{n}.$$

It is clear that $$g_{0}=e_{0}.$$

Now for $n\geq 1$, $$g_{n}=e_{n}-\sum_{k=0}^{n-1}\langle e_{n},g_{k}\rangle g_{k}=e_{n}.$$

Now suppose that $\{e_{n}\}_{n=0}^{\infty}$ is an orthonormal basis. Just like the Hilbert space setting, consider $\sum_{n=0}^{k}c_{n}e_{n}$ for some $\{c_{n}\}_{n=0}^{k}\subseteq A$ and $x\in H$.
One can show that:
$$\left \langle x-\sum_{n=0}^{k}c_{n}e_{n}, x-\sum_{n=0}^{k}c_{n}e_{n} \right\rangle=$$
$$\left\langle x-\sum_{n=0}^{k}\langle x,e_{n} \rangle e_{n}, x-\sum_{n=0}^{k}\langle x, e_{n}\rangle e_{n} \right\rangle+\left\langle \sum_{n=0}^{k}(\langle x,e_{n}\rangle -c_{n})e_{n},\sum_{n=0}^{k}(\langle x,e_{n}\rangle -c_{n})e_{n} \right\rangle .$$
Therefore, since all these elements are non-negative in $A$, we have
$$\left\|x-\sum_{n=0}^{k}\langle x,e_{n} \rangle e_{n} \right\|\leq \left\|x-\sum_{n=0}^{k}c_{n}e_{n} \right\|.$$
Therefore, since $\{e_{n}\}_{n=0}^{\infty}$ generates $H$, it follows that
$$x=\sum_{n=0}^{\infty}\langle x,e_{n}\rangle e_{n},$$ and it easy to see from Lemma \ref{auxseqform} that $\{e_{n}\}_{n=0}^{\infty}$ is effective.
\end{proof}
\begin{rem}
The basis criteria discussed in \cite{Frank2002Frames}: $$\sum_{n=0}^{\infty}c_{n}e_{n}=0 \ \text{for some} \ \{c_{n}\}_{n=0}^{\infty}\subseteq A\implies \ c_{n}e_{n}=0, \ \forall n.$$ is just as strong as the hypothesis of the previous theorem since we assume that $$\langle e_{n},e_{n}\rangle=I$$ for all $n$.

\end{rem}

Our last result of this section is an analogue of the Kaczmarz algorithm in the finite-dimensional setting. However, to prove it, we are not able to draw on a Pythagorean theorem from the Hilbert space setting. Instead, we prove that if a finite sequence algebraically generates the Hilbert $A$-module $H$, the norm of coefficients that generate elements of $H$ can be controlled, which is what we need to lift the argument to the Hilbert $C^{*}$-module setting.

\begin{lem}\label{frameoplemma}
Suppose that $A$ is a unital $C^{*}$ algebra and that $\{e_{n}\}_{n=0}^{k}$ algebraically generates the Hilbert $A$-module $H$. Then $S:H\to H$ where $$S(f)=\sum_{n=0}^{k}\langle f, e_{n}\rangle e_{n}$$ is a bounded invertible operator.
\end{lem}
\begin{proof}
Let $T: A^{k+1}\to H$ where
$$T\{c_{n}\}=\sum_{n=0}^{k}c_{n}e_{n}.$$ By hypothesis, $T$ is surjective. Now consider the family of $A$-linear bounded maps $F_{y}: H\to A$
where $$F_{y}(x)=\langle x, y\rangle $$ and $$\|\{\langle y, e_{n}\rangle \}\|_{A^{k+1}}\leq 1.$$
For each $x\in H$, there is $x'\in A^{k+1}$ such that $$Tx'=x.$$ Therefore,
$$\|\langle Tx', y\rangle \| =\|\langle x', \{\langle y,e_{n}\rangle \} \rangle_{A^{k+1}} \| \leq \|x'\|.$$

That is, the family $\{F_{y}\}$ is point-wise bounded and therefore, uniformly bounded by the uniform boundedness principle. Consequently,

$\{y\in H: \|\{\langle y, e_{n}\rangle \}\|\leq 1\}$ is a bounded set. It now follows easily that for any $x\in H$
$$T^{*}x=\{\langle x,e_{n}\rangle\},$$
and there is an $a>0$ such that 
$$a\|x\|\leq \|T^{*}x\|$$
for all $x\in H$. In particular, $Im(T^{*})$ is closed. It is also clear that $Im(T^{*})$ is finitely generated and hence self-dual, see \cite{Frank2002Frames}; therefore,
$$A^{k+1}=Im(T^{*})\oplus Im(T^{*})^{\perp}.$$ Additionally, since it easily follows that 
$$Im(T^{*})^{\perp}\subseteq ker(T),$$
$$Im(S)=Im(TT^{*})=H.$$

Now if $$S(x)=0,$$ since $S$ is self-adjoint,
$$\langle x, S(y)\rangle =0$$ for any $y\in H$.
Because $S$ is surjective, it follows that $x=0$, and $S$ is invertible.

\end{proof}

\begin{thm}\label{finite-dimcase}
Suppose that $A$ is a unital $C^{*}$ algebra and that $\{e_{n}\}_{n=0}^{k}$ algebraically generates the Hilbert $A$-module $H$. Also, suppose that $$\langle e_{n},e_{n}\rangle =I$$ for all $n$. Then $\{e_{n}\}_{n=0}^{\infty}$ is effective where $\{e_{n}\}_{n=0}^{\infty}$ is $k+1$ periodic extension of $\{e_{n}\}_{n=0}^{k}$.  
\end{thm}
\begin{proof}
For each $n$, define $P_{n}: H\to H$ where $$P_{n}(x)=\langle x,e_{n}\rangle e_{n}$$ and 
$$M:=(I-P_{k})\cdots (I-P_{0}).$$ First note that for any $x\in H$ and any $n$,
$$\langle x,x\rangle =\langle x-\langle x,e_{n}\rangle e_{n}, x-\langle x,e_{n}\rangle e_{n}\rangle+\langle x,e_{n}\rangle \langle e_{n}, x\rangle.$$
Therefore since every term above is non-negative in $A$, we have $$\|(I-P_{n})x\|\leq \|x\|.$$
Consequently, $$\|M\|\leq 1.$$
We will show that $\|M\| \neq 1,$ which will finish the proof since one can show by induction and Definition \ref{kaczmarzdef} that $$x-x_{n}=(I-P_{n})\dots (I-P_{0})x, \ \forall n\geq0.$$

Suppose for the sake of contradiction that $$\|M\| =1.$$
It follows from the proof of Theorem \ref{pframecond} that for any $\epsilon>0$, there is a unit vector $x_{\epsilon}$ such that
$$\langle x_{\epsilon}, x_{\epsilon}\rangle=\langle Mx_{\epsilon}, Mx_{\epsilon}\rangle+\sum_{n=0}^{k}\langle x_{\epsilon}, g_{n}\rangle \langle g_{n},x_{\epsilon} \rangle$$
and $$\|\langle Mx_{\epsilon}, Mx_{\epsilon}\rangle \|>(1-\epsilon)^{2}$$ where $\{g_{n}\}_{n=0}^{\infty}$ is the auxiliary sequence of $\{e_{n}\}_{n=0}^{\infty}$. Therefore by passing to $B(H')$ via the Gelfand–Naimark theorem, there is a unit $f \in H'$ such that
$$\langle \langle x_{\epsilon}, x_{\epsilon}\rangle f, f\rangle_{H'}>(1-\epsilon)^{2} ,$$
but 
$$\left\langle \sum_{n=0}^{k}\langle x_{\epsilon}, g_{n}\rangle \langle g_{n},x_{\epsilon}\rangle f, f \right\rangle_{H'}< 1-(1-\epsilon)^{2}.$$

Additionally, since $\{e_{n}\}_{n=0}^{k}$ algebraically generates $H$, then $\{g_{n}\}_{n=0}^{k}$ algebraically generates $H$ by Definition \ref{auxseq}.
Therefore, by using $S$ from Lemma \ref{frameoplemma},
$$(1-\epsilon)^{2}<\langle \langle x_{\epsilon}, x_{\epsilon}\rangle f, f\rangle_{H'}\leq \left\| \left(\sum_{n=0}^{k}\langle S^{-1}x_{\epsilon}, g_{n}\rangle \langle g_{n},x_{\epsilon}\rangle \right) f \right\|_{H'}$$
$$\leq \sum_{n=0}^{k}\|\langle S^{-1}x_{\epsilon}, g_{n}\rangle \|_{A} \|\langle g_{n},x_{\epsilon}\rangle f\|_{H'}\leq \sqrt{\sum_{n=0}^{k}\|\langle S^{-1}x_{\epsilon}, g_{n}\rangle \|_{A}^{2}} \sqrt{\sum_{n=0}^{k}\|\langle g_{n},x_{\epsilon}\rangle f \|_{H'}^{2}}$$
$$\leq \|S^{-1}\|\sqrt{\sum_{n=0}^{k}\|g_{n}\|^{2}}\sqrt{1-(1-\epsilon)^{2}}.$$

If $\epsilon$ is small enough, this is a contradiction.
\end{proof}

The following is a simple example of applying Theorem \ref{finite-dimcase} to illustrate what unit vectors in the $A$-valued inner product sense look like in a concrete setting:

\begin{ex}
Let $H'$ be a Hilbert space and $U,V,W\in B(H')$ be co-isometries. Define $$e_{0}=\begin{bmatrix} \frac{1}{\sqrt{2}}U \\ \frac{1}{\sqrt{2}}V\end{bmatrix}, e_{1}=\begin{bmatrix} W \\ 0 \end{bmatrix}.$$ Then $\{e_{n}\}_{n=0}^{\infty}\subseteq B(H')^{2}$, the $2$ periodic extension of $\{e_{0},e_{1}\}$, is effective.
\end{ex}

\section{Stationary case}\label{S4}
In this section, we focus on the stationary case. Remarkably, with some careful adjustments to their original arguments, the results of Kwapień and Mycielski for the stationary case in Hilbert spaces can, in a certain sense, be extended to the setting of Hilbert $C^{*}$-modules. However, this extension resembles the Hilbert space result more closely in the commutative setting. Additionally, the definition of stationary sequences in Hilbert spaces extends naturally to the Hilbert 
$C^{*}$-module setting, just as one would hope:

\begin{defn}
Let $A$ be a unital $C^{*}$-algebra and $H$ be a Hilbert $A$-module. We say that $\{e_{n}\}_{n=0}^{\infty}\subseteq H$ is \textbf{stationary} if there is an $A$-linear isometry $T:H\to H$ such that $$T^{n}e_{0}=e_{n}, \ \forall n.$$
\end{defn}

\begin{rem}  
Note that $T$ in the above definition is inner product preserving by Blecher \cite{Blecher1997New}.
\end{rem}

We remind the reader of Kwapień and Mycielski's results for stationary sequences here.
Also, the following is due to Saraon \cite{Sarason1994Sub-Hardy}, which Kwapien and Mycielski use in their original proof:
\begin{thm*}[Sarason]\label{sarasonthm}
Suppose that $\{\langle e_{0}, e_{n}\rangle\}_{n=0}^{\infty}\subseteq \mathbb{C}$ are the non-negative Fourier coefficients of a Borel probability measure on $[0,1)$. Define $\{c_{n}\}_{n=0}^{\infty}\subseteq \mathbb{C}$ so that $\sum_{n=0}^{\infty}c_{n}z^{n}$ is the inverse of $\sum_{n=0}^{\infty}\langle e_{0}, e_{n}\rangle z^{n}$ in $\mathbb{C}[[z]]$. 

Then $\{\langle e_{0}, e_{n}\rangle\}_{n=0}^{\infty}$ are the non-negative Fourier coefficients of a singular Borel probability measure on $[0,1)$ if and only if $$\sum_{n=1}^{\infty}|c_{n}|^{2}=1$$ if and only if
$\sum_{n=0}^{\infty}c_{n}z^{n}$ is an inner function.
\end{thm*}

\begin{thm*}[Kwapien and Mycielski]

Let $H'$ be a Hilbert space and $\{e_{n}\}_{n=0}^{\infty}\subseteq H'$ be a stationary sequence of complete unit vectors. Then $\{e_{n}\}_{n=0}^{\infty}$ is effective if and only if $\{\langle e_{0}, e_{n}\rangle \}_{n=0}^{\infty}$ are the non-negative Fourier coefficients of a singular measure or Lebesgue measure.
\end{thm*}

\begin{thm}\label{stationarythm}
Let $A$ be a unital $C^{*}$-algebra and $H$ be a Hilbert $A$-module.
Suppose that $\{e_{n}\}_{n=0}^{\infty}\subseteq H$ is a stationary sequence that generates $H$, and that $$\langle e_{0},e_{0}\rangle=I.$$ Then $\{e_{n}\}_{n=0}^{\infty}$ is effective if and only if for all $k>0$,
$$\langle e_{k}, e_{0}\rangle \sum_{n=1}^{\infty}c_{n}c_{n}^{*}\langle e_{0},e_{k}\rangle =\langle e_{k}, e_{0}\rangle\langle e_{0},e_{k}\rangle$$ where $$c_{0}=I$$ and for $n\geq 1$, $$c_{n}=-\sum_{k=0}^{n-1}c_{k}\langle e_{0}, e_{n-k}\rangle.$$
\end{thm}
\begin{proof}
We claim for all $n$, $$g_{n}=\sum_{k=0}^{n}c_{n-k}^{*}e_{k}$$
where $\{g_{n}\}$ is the auxiliary sequence of $\{e_{n}\}$.

The $n=0$ case is obvious.

Suppose the result holds for all $0\leq k\leq n-1$.
Then by the stationary condition and inductive assumption,
$$g_{n}=e_{n}-\sum_{k=0}^{n-1}\langle e_{n-k},e_{0}\rangle \sum_{j=0}^{k}c_{k-j}^{*}e_{j}$$
$$=e_{n}-\sum_{j=0}^{n-1}\left(\sum_{k=j}^{n-1}\langle e_{n-k},e_{0}\rangle c_{k-j}^{*} \right)e_{j}$$
$$=e_{n}+\sum_{j=0}^{n-1}c_{n-j}^{*}e_{j}.$$
The claim follows.

Now we claim that for $n\geq j\geq 1$,
\begin{equation}\label{nastyinduction}
\sum_{k=0}^{n}\langle e_{j},g_{k}\rangle e_{k}-e_{j}=\sum_{k=1}^{j}\langle e_{k},e_{0}\rangle \left(\sum_{m=0}^{n+k-j}c_{m}e_{m+j-k} \right).    
\end{equation}
This claim can be proved by induction and using the previous claim as well as the fact that $\sum_{n=0}^{\infty}c_{n}x^{n}$ is the inverse of $\sum_{n=0}^{\infty}\langle e_{0},e_{n}\rangle x^{n}$ in the ring $A[[x]]$, which follows from the fact that $\sum_{n=0}^{\infty}c_{n}x^{n}$ is the unique left inverse of $\sum_{n=0}^{\infty}\langle e_{0},e_{n}\rangle x^{n}$ by definition. An almost identical claim is given by Kwapien and Mycielski \cite{Kwapien2001Kaczmarz}; therefore, the proof of this claim is omitted. However, if the reader would like to see details of this claim in the scalar case, which transfers easily over to the $C^{*}$-algebra case, they can read John Herr's dissertation \cite{Herr2016Fourier}, which provides thorough and clear details of Kwapien and Mycielski's proof.

Suppose that 
$\{e_{n}\}_{n=0}^{\infty}$ is effective. If $\{e_{n}\}_{n=0}^{\infty}$ is an orthogonal sequence, the conclusion easily follows, so suppose that $\{e_{n}\}_{n=0}^{\infty}$ is not an orthogonal sequence.
Let $\{e_{n_{r}}\}_{r=0}^{R}$ be the subsequence of $\{e_{n}\}_{n=1}^{\infty}$ where 
$$\langle e_{n_{r}}, e_{0}\rangle \neq 0 $$ and $R$ is possibly infinity. We prove for all $r$,
$$\langle e_{n_{r}},e_{0}\rangle \sum_{m=0}^{\infty}c_{m}e_{m}=0$$ by induction on $r$.

For $r=0$, by equation \ref{nastyinduction} and since $\{e_{n}\}_{n=0}^{\infty}$ is effective,
$$\langle e_{n_{0}},e_{0}\rangle \sum_{m=0}^{n}c_{m}e_{m}\to 0.$$

Suppose the claim holds for all $0\leq k\leq r$. By equation \ref{nastyinduction} and since $\{e_{n}\}_{n=0}^{\infty}$ is effective,
$$\sum_{k=0}^{r+1}\langle e_{n_{k}},e_{0}\rangle \left( \sum_{m=0}^{n+k-n_{r+1}}c_{m}e_{m+n_{r+1}-n_{k}} \right) \to 0.$$
Using the inductive hypothesis and remembering that $\{e_{n}\}_{n=0}^{\infty}$ is the orbit of a continuous, $A$-linear operator on $H$, we have
$$\langle e_{n_{r+1}},e_{0}\rangle \sum_{m=0}^{n}c_{m}e_{m}\to 0.$$

Conversely, suppose that for all $r$,
$$\langle e_{n_{r}},e_{0}\rangle \sum_{m=0}^{n}c_{m}e_{m}\to 0.$$
We have seen in the proof of Theorem \ref{finite-dimcase} that $$\|x-x_{n}\|\leq \|x\|$$ for any $x\in H$ and any $n$. Then it follows from equation \ref{nastyinduction} that $\{e_{n}\}_{n=0}^{\infty}$ is effective since the sequence of bounded $A$-linear maps $x\to x-x_{n}$ converge to zero point-wise on an $A$-linear dense subset of $H$, and they are uniformly bounded, so the sequence converges to zero in strong operator topology.
Thus, we have proved that $\{e_{n}\}_{n=0}^{\infty}$ is effective if and only if $$\langle e_{k},e_{0}\rangle \sum_{m=0}^{\infty}c_{m}e_{m}=0$$ for all $k>0$.

We finish the proof by showing that for all $n>0$,
$$\left \langle \sum_{m=0}^{n}c_{m}e_{m}, \sum_{m=0}^{n}c_{m}e_{m} \right\rangle =I-\sum_{m=1}^{n}c_{m}c_{m}^{*}.$$
The $n=1$ case is easy to check.

Suppose the result holds for all $2\leq k<n$.
Then we have
$$\left \langle \sum_{m=0}^{n}c_{m}e_{m}, \sum_{m=0}^{n}c_{m}e_{m} \right\rangle=$$
$$I-\sum_{m=1}^{n-1}c_{m}c_{m}^{*}+c_{n}c_{n}^{*}+\left\langle \sum_{m=0}^{n-1}c_{m}e_{m}, c_{n}e_{n} \right\rangle + \left\langle c_{n}e_{n}, \sum_{m=0}^{n-1}c_{m}e_{m} \right\rangle=$$
$$I-\sum_{m=1}^{n-1}c_{m}c_{m}^{*}+c_{n}c_{n}^{*}+\left(\sum_{m=0}^{n-1}c_{m}\langle e_{0},e_{n-m}\rangle \right)c_{n}^{*}+c_{n}\left(\sum_{m=0}^{n-1}c_{m}\langle e_{0},e_{n-m}\rangle \right)^{*}=$$
$$I-\sum_{m=1}^{n-1}c_{m}c_{m}^{*}+c_{n}c_{n}^{*}-c_{n}c_{n}^{*}-c_{n}c_{n}^{*}.$$

\end{proof}

The following corollary establishes that when $A$ is commutative, a stationary sequence of unit vectors (with respect to the inner product) that generate a Hilbert 
$A$-module is effective if and only if each fiber of the inner product space is the sequence of Fourier coefficients of either a singular measure or the Lebesgue measure. In particular, it suffices for each fiber of the inner product space to coincide with that of a singular or Lebesgue measure to ensure uniform recovery via the Kaczmarz algorithm. This characterization follows from the Gelfand–Naimark theorem, which identifies any commutative unital $C^{*}$-algebra with the algebra of continuous functions on a compact Hausdorff space, together with Dini’s theorem to ensure uniform recovery.
\begin{cor}\label{stationarycomm}
Suppose the hypothesis of Theorem \ref{stationarythm} and that $A$ is commutative. Let $\phi: A\to C(X)$ be the isomorphism via the Gelfand-Naimark theorem. Then for every $x\in X$, $\{\phi \langle e_{0},e_{n}\rangle(x)\}$ are the non-negative Fourier coefficients of a Borel probability measure on $[0,1)$. Furthermore, the following are equivalent:
\begin{enumerate}
    \item $\{e_{n}\}_{n=0}^{\infty}$ is effective.
    \item For each $x\in X$, $\{\phi \langle e_{0},e_{n}\rangle(x)\}$ are the non-negative Fourier coefficients of a singular measure or Lebesgue measure.
\end{enumerate}
In particular, if for each $x\in X$,
$\sum_{n=0}^{\infty}\phi c_{n}(x)z^{n}$ is an inner function, then $\{e_{n}\}_{n=0}^{\infty}$ is effective.
\end{cor}
\begin{proof}
For any $x\in X$, one can easily check that by the stationary condition on $\{e_{n}\}_{n=0}^{\infty}$ that $\psi_{x}:\mathbb{Z}\to \mathbb{C}$ where $$\psi_{x}(n)=\phi\langle e_{k},e_{k+n}\rangle(x)$$ is a positive definite function that maps $0$ to $1$ where $k\geq max\{0,-n\}$. Then by Bochner's theorem, for every $x\in X$, $\{\phi \langle e_{0},e_{n}\rangle(x)\}$ are the non-negative Fourier coefficients of a Borel probability measure on $[0,1)$. 

Furthermore, it is clear that for each $x\in X$, $\sum_{n=0}^{\infty}\phi \langle e_{0},e_{n}\rangle(x)z^{n}$ and $\sum_{n=0}^{\infty}\phi c_{n}(x)z^{n}$ are inverses in the ring $\mathbb{C}[[z]]$ by the definition of the $\{c_{n}\}_{n=0}^{\infty}$ from Theorem \ref{stationarythm}.

Now suppose that $\{e_{n}\}_{n=0}^{\infty}$ is effective. Then by Theorem \ref{stationarythm} for each $x\in X$ where for some $n>0$, $$\phi \langle e_{0},e_{n}\rangle(x)\neq 0,$$
$$\sum_{n=1}^{\infty}|\phi c_{n}(x)|^{2}=1.$$ By Sarason, such $x$ have $\{\phi \langle e_{0},e_{n}\rangle(x)\}$ being the non-negative Fourier coefficients of a singular Borel probability measure on $[0,1)$. For $x$ otherwise, clearly $\{\phi \langle e_{0},e_{n}\rangle(x)\}$ are the non-negative Fourier coefficients of Lebesgue measure.

Now suppose that for each $x$, $\{\phi \langle e_{0},e_{n}\rangle(x)\}$ are the non-negative Fourier coefficients of a singular measure or Lebsgue measure. It follows that for all $k>0$,
$$|\phi \langle e_{k}, e_{0}\rangle|^{2} \sum_{n=1}^{\infty}|\phi c_{n}|^{2} =|\phi \langle e_{k}, e_{0}\rangle|^{2}$$ point-wise.
This convergence is uniform in $X$, by Dini's theorem. Therefore, $\{e_{n}\}_{n=0}^{\infty}$ is effective by Theorem \ref{stationarythm}.
\end{proof}

The following application of Corollary \ref{stationarycomm} shows that a continuous family of elements in a Hilbert space can be recovered uniformly by an effective stationary sequence:

\begin{cor}\label{uniformkacz}
Let $H'$ be a Hilbert space and $X$ be a compact Hausdorff space. Then $H=C(X,H')$ is a Hilbert $C(X)$-module. If $\{e_{n}\}_{n=0}^{\infty}\subseteq H'$ is a stationary sequence of complete unit vectors where $\{\langle e_{0}, e_{n}\rangle\}_{n=0}^{\infty}$ are the non-negative Fourier coefficients of a singular Borel probability measure, then $\{e_{n}\}_{n=0}^{\infty}\subseteq H$ is effective where $\{e_{n}\}_{n=0}^{\infty}\subseteq H$ denotes elements of $C(X,H')$ that are constant in $X$.
\end{cor}
\begin{proof}
We show that $\{e_{n}\}_{n=0}^{\infty}$ generates $H$. Let $f(x)\in H$ and $\epsilon>0$. By continuity of $f$ and compactness of $X$, there is finitely many $x_{k}\in X$ s.t. $f(X)\subseteq \bigcup_{k}B(f(x_{k}),\epsilon)$. Therefore, $\left\{\bigcup_{k}f^{-1}B(f(x_{k}),\epsilon) \right\}_{k=0}^{m}$ is a finite open cover of $X$. Since $X$ is normal, there is a partition of unity $\{\phi_{k}\}$ dominated by $\left\{\bigcup_{k}f^{-1}B(f(x_{k}),\epsilon) \right\}_{k=0}^{m}$. Now for each $k$, let $\psi_{k}\in span\{e_{n}\}$ such that
$$\|\psi_{k}-f(x_{k})\|_{H'}<\epsilon.$$
Additionally, for any $x\in X$, let $k_{j}\in \{0,\dots, m\}$ be such that $\phi_{k_{j}}(x)\neq 0$. Then $$\|f(x)-\psi_{k_{j}}\|_{H'}<2\epsilon$$ for all $j$.
Then we have
$$\left\|\sum_{k=0}^{m}\phi_{k}(x)\psi_{k}-f(x) \right\|_{H'}=\left\|\sum_{j=0}^{r}\phi_{k_{j}}(x)\psi_{k_{j}}-\sum_{j=0}^{r}\phi_{k_{j}}(x)f(x) \right\|_{H'}< 2\epsilon.$$
\end{proof}

Now for a more interesting example, the following space is constructed via continuous fields of Hilbert spaces, which are defined by Douglas and Paulsen in \cite{Douglas1989Hilbert}.

\begin{defn}
Let $X$ be a compact Hausdorff space and $\{\mu_{x}\}_{x\in X}$ be a weakly continuous family of Borel probability measures on $[0,1)$. That is, for any $g\in C_{b}([0,1))$, $x\to \int_{0}^{1}gd\mu_{x}$ is continuous.
Define the inner product on $$V=span_{C(X)}\{e^{2\pi i ny}\}_{n=0}^{\infty}\subseteq \left\{f:X \to \bigcup L^{2}(\mu_{x}) \ . \ \forall x\ f(x)\in L^{2}(\mu_{x}) \right\}$$ as follows:
$$\left\langle \sum_{n}c_{n}e^{2\pi i ny},\sum_{n}d_{n}e^{2\pi i ny} \right\rangle:=\sum_{n}\sum_{k}c_{n}\overline{d_{k}}\langle e^{2\pi i ny}, e^{2\pi i ky}\rangle_{\mu_{x}}. $$
Denote the norm completion of $V$ by $WC(X, L^{2}(\mu_{x}))$. Then $WC(X, L^{2}(\mu_{x}))$ is a Hilbert $C(X)$ module.    
\end{defn}

\begin{cor}\label{contcor}
 $\{e^{2\pi i ny}\}_{n=0}^{\infty}\subseteq WC(X, L^{2}(\mu_{x}))$ is effective if and only if for every $x\in X$, $\mu_{x}$ is singular or Lebesgue measure.
\end{cor}

\begin{ex}
Suppose \( \mu \) is a fixed singular Borel probability measure on \([0,1)\), and let $$\mu_x = g(x)\mu,$$ where 
\( g \in C\big(X, C_b([0,1))\big) \) and \( g(x)(y) \geq 0 \) for all \( x \in X \) and \( y \in [0,1) \), with the normalization condition
\[
\int g(x)(y)\, d\mu(y) = 1
\]
for each \( x \in X \). Then \( \{ \mu_x \}_{x \in X} \) forms a weakly continuous family of singular probability measures. One can also check that the standard module \( C(X, L^2(\mu)) \) embeds into \( WC(X, L^2(\mu_x)) \). Consequently, any continuous family of functions in \( L^2(\mu) \) can be uniformly recovered in the varying Hilbert spaces \( L^2(\mu_x) \) via the Kaczmarz algorithm.
\end{ex}

\section{A normalized Cauchy transform for continuous families of measures}\label{S5}
Our goal for this section is to find the image of the following operator:
\begin{defn}
Let $V_{\mu_{\{x\}}}: WC(X, L^{2}(\mu_{x}))\to C(X, H^{2}(\mathbb{D}))$ where
$$V_{\mu_{\{x\}}}= \sum_{n=0}^{\infty}\langle f, g_{n}\rangle w^{n}$$
and $\{g_{n}\}$ is the auxiliary sequence of $\{e^{2\pi i ny}\}_{n=0}^{\infty}\subseteq WC(X, L^{2}(\mu_{x}))$ where for every $x\in X$, $\mu_{x}$ is singular or Lebesgue measure.   
\end{defn}
We will attempt to convince the reader that $V_{\mu_{\{x\}}}$ is the normalized Cauchy transform of a continuous family of measures.
We recall the classical normalized Cauchy transform and the Herglotz representation theorem here:
\begin{thm*}[Herglotz]
There is a one-to-one correspondence between Borel probability measures on $[0,1)$ $\mu$, and analytic functions $b(w)$ bounded by one on $\mathbb{D}$ such that $b(0)=0$ where
$$\frac{1+b(w)}{1-b(w)}=\int_{0}^{1}\frac{1+we^{-2\pi i x}}{1-we^{-2\pi i x}}d\mu.$$

In particular, there is a one-to-one correspondence between singular Borel probability measures on $[0,1)$ $\mu$, and inner functions $b(w)$ such that $b(0)=0$.
\end{thm*}

\begin{defn}
Let $\mu$ be a finite Borel measure on $[0,1)$. Recall that the \textbf{Cauchy transform}, $C_{\mu}$, is a map from $L^{1}(\mu)$ to analytic functions on the disk such that
$$[C_{\mu}(f)](w)=\int_{0}^{1}\frac{f(x)d\mu}{1-we^{-2\pi i x}}.$$
Now specifically for $\mu$ that is a singular Borel probability measure on $[0,1)$ (or Lebesgue measure), let $b(w)$ be the corresponding function of $\mu$ via the Herglotz Theorem.
Define $V_{\mu}: L^{2}(\mu)^{+}\to [b(w)H^{2}(\mathbb{D})]^{\perp}$
by \begin{equation}V_{\mu}(f)=\frac{C_{\mu}(f)}{C_{\mu}(1)} = \int_{0}^{1} \dfrac{f(x) d\mu}{(1 - (\cdot)e^{-2\pi i x}) [C_{\mu}(1)](\cdot)},\end{equation}
which is a unitary operator called the \textbf{normalized Cauchy transform} \cite{Sarason1994Sub-Hardy, Clark1972One}.
\end{defn}

\begin{lem*}[Herr and Weber \cite{Herr2017Fourier}]
Let $\mu$ be a Borel probability measure on $[0,1)$ and $b(w)$ be the corresponding inner function via the Herglotz theorem. Then
\begin{equation}{\label{inner}}
b(w)=1-\frac{1}{[C_{\mu}(1)](w)}.
\end{equation}
\end{lem*}

The following result establishes the connection between the normalized Cauchy transform and auxiliary sequences from the Kaczmarz algorithm:

\begin{thm*}[Herr and Weber]
Let $\mu$ be a singular Borel probability measure on $[0,1)$ with auxiliary sequence $\{g_{n}\}_{n=0}^{\infty}$ of $\{e^{2\pi i nx}\}_{n=0}^{\infty}\subseteq L^{2}(\mu)$.
Then \begin{equation}{\label{NCT}}[V_{\mu}(f)](w)=\sum_{n=0}^{\infty}\langle f,g_{n}\rangle w^{n}.\end{equation}
\end{thm*}

Note that $V_{\mu_{\{x\}}}$ is an isometry by Corollary \ref{contcor} and Theorem \ref{pframecond}. Furthermore, by the previous theorem, for $x\in X$, $[V_{\mu_{\{x\}}}(f)](x)$ is the normalized Cauchy transform of $f(x)\in L^{2}(\mu_{x})$. Therefore, $V_{\mu_{\{x\}}}$ represents a normalized Cauchy transform for a continuous family of measures. What follows is our main result for the image of $V_{\mu_{\{x\}}}$, which resembles a continuous version of backward-invariant subspaces of $H^{2}(\mathbb{D})$, see \cite{Beurling1948Two}:

\begin{thm}\label{imageNCT}
$$ImV_{\mu_{\{x\}}}=[b^{x}(w)C(X, H^{2}(\mathbb{D}))]^{\perp}$$ where $\{b^{x}(w)\}_{x\in X}$ is the family of inner functions and zero functions corresponding to $\{\mu_{x}\}_{x\in X}$ via the Herglotz theorem.
\end{thm}
\begin{proof}
Recall that for each $x\in X$,
$$b^{x}(w)=\sum_{n=1}^{\infty}c_{n}(x)w^{n}=1-\frac{1}{[C_{\mu_{x}}(1)](w)}.$$ Therefore, since $\{\mu_{x}\}$ is a weakly continuous family of measures, for each $n$, $c_{n}\in C(X)$.
Furthermore, by the Herglotz theorem, for any $x\in X$, $$\|b^{x}(w)f(x)\|_{H^{2}(\mathbb{D})}\leq \|f(x)\|_{H^{2}(\mathbb{D})}$$ whenever $f(x)\in C(X, H^{2}(\mathbb{D}))$. Therefore,
$$b^{x}(w)C(X, H^{2}(\mathbb{D}))\subseteq C(X, H^{2}(\mathbb{D}))$$ via point-wise multiplication of elements of $H^{2}(\mathbb{D})$.

Suppose that $\sum_{n=0}^{\infty}f_{n}(x)w^{n}\in [b^{x}(w)C(X, H^{2}(\mathbb{D}))]^{\perp}$. Then for each $x\in X$, $\sum_{n=0}^{\infty}f_{n}(x)w^{n}\in [b^{x}(w)H^{2}(\mathbb{D})]^{\perp}$. Therefore, for all $x\in X$, there is a $g(x)\in L^{2}(\mu_{x})^{+}$ such that
$$\sum_{n=0}^{\infty}f_{n}(x)w^{n}=\sum_{n=0}^{\infty}\langle g(x), g_{n}^{x}\rangle_{\mu_{x}}w^{n}$$ where $\{g_{n}^{x}\}$ is the auxiliary sequence of $\{e^{2\pi i ny}\}_{n=0}^{\infty}\subseteq L^{2}(\mu_{x})$. Now for each $x\in X$,
\begin{equation}\label{NCTfibers}
\sum_{n=0}^{\infty}|\langle g(x), g_{n}^{x}\rangle_{\mu_{x}}|^{2}=\|g(x)\|^{2}_{L^{2}(\mu_{x})}
\end{equation}
since each normalized Cauchy transform is a unitary.
Furthermore, by Dini's theorem, $\sum_{n=0}^{\infty}|f_{n}(x)|^{2}$ converges uniformly in $X$; therefore, equation \ref{NCTfibers} also converges uniformly in $X$. Then by the proof of Theorem \ref{pframecond}, $g: X\to L^{2}(\mu_{x})$ is generated by $\{e^{2\pi i ny}\}_{n=0}^{\infty}\subseteq WC(X, L^{2}(\mu_{x}))$. Therefore, 
$$ImV_{\mu_{\{x\}}}\supseteq[b^{x}(w)C(X, H^{2}(\mathbb{D}))]^{\perp}.$$

Now assume $\sum_{n=0}^{\infty}\langle g(x), g_{n}^{x}\rangle_{\mu_{x}}w^{n}\in ImV_{\mu_{\{x\}}}$ and $\sum_{n=0}^{\infty}f_{n}(x)w^{n}\in C(X, H^{2}(\mathbb{D}))$. Then we have,
$$\left\langle \sum_{n=0}^{\infty}\langle g(x), g_{n}^{x}\rangle_{\mu_{x}}w^{n}, b^{x}(w)\sum_{n=0}^{\infty}f_{n}(x)w^{n} \right\rangle_{C(X, H^{2}(\mathbb{D}))}=0$$ since $\sum_{n=0}^{\infty}\langle g(x), g_{n}^{x}\rangle_{\mu_{x}}w^{n}\in [b^{x}(w)H^{2}(\mathbb{D})]^{\perp}$ for each $x\in X$.
\end{proof}
As a result, every standard frame in a Hilbert $C(X)$-module whose analysis operator has the same image as $V_{\mu_{\{x\}}}$ is similar to the auxiliary sequence associated with the effective sequence $\{e^{2\pi i n y}\}_{n=0}^{\infty}\subseteq WC(X, L^2(\mu_{x})$. We define the analysis operator of a standard frame in a Hilbert $C(X)$-module that will make our analysis convenient:

\begin{defn}
Let $\{f_{n}\}_{n=0}^{\infty}$ be a standard frame in a Hilbert $C(X)$-module $H$ where $X$ is a compact Hausdorff space. Define $\theta_{g_{n}}: H\to C(X, H^{2}(\mathbb{D}))$ where
$$\theta_{g_{n}}f=\sum_{n=0}^{\infty}\langle f, g_{n}\rangle w^{n}.$$
\end{defn}

\begin{cor}\label{equivalenttoauxcor}
Let $\{f_{n}\}_{n=0}^{\infty}$ be a standard frame in a Hilbert $C(X)$-module $H$ where $X$ is a compact Hausdorff space. The following are equivalent:

\begin{enumerate}
    \item $$Im\theta_{f_{n}}=[b^{x}(w)C(X, H^{2}(\mathbb{D}))]^{\perp}$$ where for each $x\in X$, $b^{x}(w)$ is an inner function such that $$b^{x}(0)=0$$ or $$b^{x}(w)=0,$$ and the family of measures corresponding to these functions via the Herglotz theorem $\{\mu_{x}\}_{x\in X}$ are weakly continuous.
    \item There is a bounded, invertible, operator with an invertible adjoint that sends $\{f_{n}\}_{n=0}^{\infty}$ to the auxiliary sequence of effective sequence $\{e^{2\pi i ny}\}_{n=0}^{\infty}\subseteq WC(X, L^{2}(\mu_{x}))$.
\end{enumerate}

\end{cor}
\begin{proof}
$(1)\implies (2)$

By Theorem \ref{imageNCT}, for all $f\in H$, there is a $B(f)\in WC(X, L^{2}(\mu_{x}))$ such that
\begin{equation}\label{ipeq}
\langle f, f_{n}\rangle_{H} =\langle B(f), g_{n}\rangle_{WC(X, L^{2}(\mu_{x}))} 
\end{equation} for all $n$ where $\{g_{n}\}$ is the auxiliary sequence of $\{e^{2\pi i ny}\}\subseteq WC(X, L^{2}(\mu_{x}))$. Also, clearly the function $B: H\to WC(X, L^{2}(\mu_{x}))$ is invertible. Now note that since $\{g_{n}\}$ is a standard Parseval frame and $\{f_{n}\}$ is a stadnard frame that $B$ is $A$-linear and bounded.

Now define $B^{*}: WC(X, L^{2}(\mu_{x}))\to H$ where $B^{*}(g_{n})=f_{n}$ for all $n$ and extend $C(X)$-linearly. To see that $B^{*}$ is well defined,
suppose that
$$\sum_{n=0}^{k}c_{n}f_{n}=0$$ for some $\{c_{n}\}\subseteq C(X)$. By equation \ref{ipeq} since $B$ is surjective,
$$\sum_{n=0}^{k}c_{n}g_{n}=0.$$ Therefore, $B^{*}$ is well defined on $span_{C(X)}\{g_{n}\}$.
Furthermore for any $\{c_{n}\}\subseteq C(X)$ by equation \ref{ipeq},
$$\left \langle \sum_{n=0}^{k}c_{n}f_{n}, \sum_{n=0}^{k}c_{n}f_{n} \right\rangle_{H}=\left\langle B\sum_{n=0}^{k}c_{n}f_{n},\sum_{n=0}^{k}c_{n}g_{n} \right\rangle_{WC(X, L^{2}(\mu_{x}))}.$$
Thus, $B^{*}$ is also bounded on $span_{C(X)}\{g_{n}\}$ since $B$ is bounded, showing $B^{*}$ uniquely extends continuously to $WC(X, L^{2}(\mu_{x}))$.

Now suppose that $$B^{*}f=0=B^{*}\sum_{n=0}^{\infty}\langle f, g_{n}\rangle_{WC(X, L^{2}(\mu_{x}))}g_{n}=\sum_{n=0}^{\infty}\langle f, g_{n}\rangle_{WC(X, L^{2}(\mu_{x}))}f_{n}=$$
$$\sum_{n=0}^{\infty}\langle B^{-1}f, f_{n}\rangle_{H}f_{n}.$$
Since $\{f_{n}\}$ is a frame, we have $f=0$.

Now to see that $B^{*}$ is surjective, note that for any $f\in H$, there is a $\{c_{n}\}\in Im\theta_{f_{n}}$ such that
$$\sum_{n=0}^{\infty}c_{n}f_{n}=f=B^{*}\sum_{n=0}^{\infty}c_{n}g_{n}.$$ The above equality is sound because $\sum_{n=0}^{\infty}c_{n}g_{n}$ converges for any $\{c_{n}\}\in Im\theta_{g_{n}}$ by the Parseval frame condition.

Now it is clear that $B^{*}$ is the adjoint of $B$.

$(2)\implies (1)$

This follows easily from Theorem \ref{imageNCT}.
\end{proof}

Our next result is that every standard Parseval frame in a Hilbert $C(X)$-module whose analysis operator has the same image as $V_{\mu_{\{x\}}}$ arises from an effective Kaczmarz algorithm. This is inspired by the result of Szwarc \cite{Szwarc2007Kaczmarz} about when Parseval frames arise from effective Kaczmarz algorithms in the Hilbert space setting.
\begin{cor}
Let $\{f_{n}\}_{n=0}^{\infty}$ be a standard Parseval frame in a Hilbert $C(X)$-module $H$ where $X$ is a compact Hausdorff space. If $$Im\theta_{f_{n}}=[b^{x}(w)C(X, H^{2}(\mathbb{D}))]^{\perp}$$ where for each $x\in X$, $b^{x}(w)$ is an inner function such that $$b^{x}(0)=0$$ or $$b^{x}(w)=0,$$ and the family of measures corresponding to these functions via the Herglotz theorem $\{\mu_{x}\}_{x\in X}$ are weakly continuous, then $\{f_{n}\}_{n=0}^{\infty}$ is the auxiliary sequence of an effective stationary sequence.
\end{cor}
\begin{proof}
Following the proof of $(1)\implies (2)$ in Corollary \ref{equivalenttoauxcor}, since we further assume that $\{f_{n}\}$ is a Parseval frame, it is unitarly equivalent to the auxiliary sequence via some operator $U$. Also, it is easy to check that $\{f_{n}\}$ is the auxiliary sequence of $\{U^{*}e^{2\pi i ny}\}_{n=0}^{\infty}\subseteq H$ using the recursive definition of auxiliary sequences. Finally, it is clear that $\{U^{*}e^{2\pi i ny}\}_{n=0}^{\infty}$ is stationary, and it is effective by Theorem \ref{pframecond}.
\end{proof}
Finally, we conclude by showing that every standard frame in a Hilbert $C(X)$-module whose analysis operator has the same image as $V_{\mu_{\{x\}}}$ is the orbit of an operator, which resembles the results in the Hilbert space setting, see \cite{Christensen2020Frame} for details.
\begin{cor}\label{orbitframeresult}
Let $\{f_{n}\}_{n=0}^{\infty}$ be a standard frame in a Hilbert $C(X)$-module $H$ where $X$ is a compact Hausdorff space. If $$Im\theta_{f_{n}}=[b^{x}(w)C(X, H^{2}(\mathbb{D}))]^{\perp}$$ where for each $x\in X$, $b^{x}(w)$ is an inner function such that $$b^{x}(0)=0$$ or $$b^{x}(w)=0,$$ and the family of measures corresponding to these functions via the Herglotz theorem $\{\mu_{x}\}_{x\in X}$ are weakly continuous, then there is a bounded operator $T:H \to H$ such that 
$$Tf_{n}=f_{n+1}$$ for all $n$.
\end{cor}
\begin{proof}
Let $\{g_{n}\}$ be the auxiliary sequence of $\{e^{2\pi i ny}\}_{n=0}^{\infty}\subseteq WC(X, L^{2}(\mu_{x}))$. It can be checked by the stationary condition that $$g_{n}=T^{n}g_{0}$$ for all $n$ where $T:WC(X, L^{2}(\mu_{x}))\to WC(X, L^{2}(\mu_{x})) $ is given by
$$Tf=e^{2\pi i y}f-\langle e^{2\pi i y}f, 1\rangle_{WC(X, L^{2}(\mu_{x}))} .$$
Since $T$ is clearly bounded, the result follows from Corollary \ref{equivalenttoauxcor}.
\end{proof}

Unfortunately, we do not believe the converse of Corollary \ref{orbitframeresult} holds, even though an analogous converse does hold in the Hilbert space setting due to Beurling \cite{Beurling1948Two}.

\printbibliography

@article{Kwapien2001Kaczmarz,
author = {Stanisław Kwapień and Jan Mycielski},
journal = {Studia Mathematica},
language = {eng},
number = {1},
pages = {75-86},
title = {On the Kaczmarz algorithm of approximation in infinite-dimensional spaces},
url = {http://eudml.org/doc/284583},
volume = {148},
year = {2001},
}

@Article{Herr2017Fourier,
AUTHOR = {Herr, John E. and Weber, Eric S.},
TITLE = {Fourier Series for Singular Measures},
JOURNAL = {Axioms},
VOLUME = {6},
YEAR = {2017},
NUMBER = {2},
ARTICLE-NUMBER = {7},
URL = {https://www.mdpi.com/2075-1680/6/2/7},
ISSN = {2075-1680},
DOI = {10.3390/axioms6020007}
}

@book {Sarason1994Sub-Hardy,
    AUTHOR = {Sarason, Donald},
     TITLE = {Sub-{H}ardy {H}ilbert spaces in the unit disk},
    SERIES = {University of Arkansas Lecture Notes in the Mathematical
              Sciences},
    VOLUME = {10},
      NOTE = {A Wiley-Interscience Publication},
 PUBLISHER = {John Wiley \& Sons, Inc., New York},
      YEAR = {1994},
     PAGES = {xvi+95},
      ISBN = {0-471-04897-6},
   MRCLASS = {46E22 (30D55 30H05 46E20 47A45 47B35 47B38)},
  MRNUMBER = {1289670},
MRREVIEWER = {Thomas\ Kriete},
}

@article {Haller2005Kaczmarz,
    AUTHOR = {Haller, Rainis and Szwarc, Ryszard},
     TITLE = {Kaczmarz algorithm in {H}ilbert space},
   JOURNAL = {Studia Math.},
  FJOURNAL = {Studia Mathematica},
    VOLUME = {169},
      YEAR = {2005},
    NUMBER = {2},
     PAGES = {123--132},
      ISSN = {0039-3223,1730-6337},
   MRCLASS = {41A65 (46C99)},
  MRNUMBER = {2140451},
MRREVIEWER = {Grzegorz\ Lewicki},
       DOI = {10.4064/sm169-2-2},
       URL = {https://doi.org/10.4064/sm169-2-2},
}

@article {Aboud2020Dual,
    AUTHOR = {Aboud, Anna and Curl, Emelie and Harding, Steven N. and
              Vaughan, Mary and Weber, Eric S.},
     TITLE = {The dual {K}aczmarz algorithm},
   JOURNAL = {Acta Appl. Math.},
  FJOURNAL = {Acta Applicandae Mathematicae},
    VOLUME = {165},
      YEAR = {2020},
     PAGES = {133--148},
      ISSN = {0167-8019,1572-9036},
   MRCLASS = {65Fxx (42C15 47A50)},
  MRNUMBER = {4058119},
       DOI = {10.1007/s10440-019-00244-6},
       URL = {https://doi.org/10.1007/s10440-019-00244-6},
}

@article {Aboud2022Kaczmarz,
    AUTHOR = {Aboud, Anna and Weber, Eric S.},
     TITLE = {The {K}aczmarz algorithm in {B}anach spaces},
   JOURNAL = {Complex Anal. Oper. Theory},
  FJOURNAL = {Complex Analysis and Operator Theory},
    VOLUME = {16},
      YEAR = {2022},
    NUMBER = {5},
     PAGES = {Paper No. 63, 32},
      ISSN = {1661-8254,1661-8262},
   MRCLASS = {41A65 (42C15 65J05)},
  MRNUMBER = {4434257},
MRREVIEWER = {Abderrazek\ Karoui},
       DOI = {10.1007/s11785-022-01246-3},
       URL = {https://doi.org/10.1007/s11785-022-01246-3},
}

@article {Christensen2020Frame,
    AUTHOR = {Christensen, Ole and Hasannasab, Marzieh and Philipp,
              Friedrich},
     TITLE = {Frame properties of operator orbits},
   JOURNAL = {Math. Nachr.},
  FJOURNAL = {Mathematische Nachrichten},
    VOLUME = {293},
      YEAR = {2020},
    NUMBER = {1},
     PAGES = {52--66},
      ISSN = {0025-584X,1522-2616},
   MRCLASS = {42C15 (30J05 30J10 47A05 94A20)},
  MRNUMBER = {4060361},
MRREVIEWER = {Timofey\ V.\ Rodionov},
       DOI = {10.1002/mana.201800344},
       URL = {https://doi.org/10.1002/mana.201800344},
}

@article {Szwarc2007Kaczmarz,
    AUTHOR = {Szwarc, Ryszard},
     TITLE = {Kaczmarz algorithm in {H}ilbert space and tight frames},
   JOURNAL = {Appl. Comput. Harmon. Anal.},
  FJOURNAL = {Applied and Computational Harmonic Analysis. Time-Frequency
              and Time-Scale Analysis, Wavelets, Numerical Algorithms, and
              Applications},
    VOLUME = {22},
      YEAR = {2007},
    NUMBER = {3},
     PAGES = {382--385},
      ISSN = {1063-5203,1096-603X},
   MRCLASS = {42C15},
  MRNUMBER = {2311862},
MRREVIEWER = {Ilya\ A.\ Krishtal},
       DOI = {10.1016/j.acha.2006.11.001},
       URL = {https://doi.org/10.1016/j.acha.2006.11.001},
}

@article {Kaczmarz1993Approximate,
    AUTHOR = {Kaczmarz, S.},
     TITLE = {Approximate solution of systems of linear equations},
      NOTE = {Translated from the German},
   JOURNAL = {Internat. J. Control},
  FJOURNAL = {International Journal of Control},
    VOLUME = {57},
      YEAR = {1993},
    NUMBER = {6},
     PAGES = {1269--1271},
      ISSN = {0020-7179,1366-5820},
   MRCLASS = {01A60 (01A70 93E35)},
  MRNUMBER = {1220361},
       DOI = {10.1080/00207179308934446},
       URL = {https://doi.org/10.1080/00207179308934446},
}

@article {Clark1972One,
    AUTHOR = {Clark, Douglas N.},
     TITLE = {One dimensional perturbations of restricted shifts},
   JOURNAL = {J. Analyse Math.},
  FJOURNAL = {Journal d'Analyse Math\'ematique},
    VOLUME = {25},
      YEAR = {1972},
     PAGES = {169--191},
      ISSN = {0021-7670,1565-8538},
   MRCLASS = {47A55},
  MRNUMBER = {301534},
MRREVIEWER = {S.\ R.\ Caradus},
       DOI = {10.1007/BF02790036},
       URL = {https://doi.org/10.1007/BF02790036},
}

@article {Czaja2013Kaczmarz,
    AUTHOR = {Czaja, Wojciech and Tanis, James H.},
     TITLE = {Kaczmarz algorithm and frames},
   JOURNAL = {Int. J. Wavelets Multiresolut. Inf. Process.},
  FJOURNAL = {International Journal of Wavelets, Multiresolution and
              Information Processing},
    VOLUME = {11},
      YEAR = {2013},
    NUMBER = {5},
     PAGES = {1350036, 13},
      ISSN = {0219-6913,1793-690X},
   MRCLASS = {42C15 (41A65 65J10)},
  MRNUMBER = {3117886},
MRREVIEWER = {Aleksandr\ Krivoshein},
       DOI = {10.1142/S0219691313500367},
       URL = {https://doi.org/10.1142/S0219691313500367},
}

@article {Beurling1948Two,
    AUTHOR = {Beurling, Arne},
     TITLE = {On two problems concerning linear transformations in {H}ilbert
              space},
   JOURNAL = {Acta Math.},
  FJOURNAL = {Acta Mathematica},
    VOLUME = {81},
      YEAR = {1948},
     PAGES = {239--255},
      ISSN = {0001-5962,1871-2509},
   MRCLASS = {46.0X},
  MRNUMBER = {27954},
MRREVIEWER = {B.\ de Sz. Nagy},
       DOI = {10.1007/BF02395019},
       URL = {https://doi.org/10.1007/BF02395019},
}

@article {Frank2002Frames,
    AUTHOR = {Frank, Michael and Larson, David R.},
     TITLE = {Frames in {H}ilbert {$C^\ast$}-modules and
              {$C^\ast$}-algebras},
   JOURNAL = {J. Operator Theory},
  FJOURNAL = {Journal of Operator Theory},
    VOLUME = {48},
      YEAR = {2002},
    NUMBER = {2},
     PAGES = {273--314},
      ISSN = {0379-4024,1841-7744},
   MRCLASS = {42C15 (46B15 46H25 46L08)},
  MRNUMBER = {1938798},
MRREVIEWER = {Mohammad\ Sal\ Moslehian},
}

@article {Blecher1997New,
    AUTHOR = {Blecher, David P.},
     TITLE = {A new approach to {H}ilbert {$C^*$}-modules},
   JOURNAL = {Math. Ann.},
  FJOURNAL = {Mathematische Annalen},
    VOLUME = {307},
      YEAR = {1997},
    NUMBER = {2},
     PAGES = {253--290},
      ISSN = {0025-5831,1432-1807},
   MRCLASS = {46L05 (46C50 46H25)},
  MRNUMBER = {1428873},
MRREVIEWER = {Christian\ Le Merdy},
       DOI = {10.1007/s002080050033},
       URL = {https://doi.org/10.1007/s002080050033},
}

@book {Herr2016Fourier,
    AUTHOR = {Herr, John Edward},
     TITLE = {Fourier series for singular measures and the {K}aczmarz
              algorithm},
      NOTE = {Thesis (Ph.D.)--Iowa State University},
 PUBLISHER = {ProQuest LLC, Ann Arbor, MI},
      YEAR = {2016},
     PAGES = {116},
      ISBN = {978-1339-84560-9},
   MRCLASS = {99-05},
  MRNUMBER = {3553554},
       URL =
              {http://gateway.proquest.com/openurl?url_ver=Z39.88-2004&rft_val_fmt=info:ofi/fmt:kev:mtx:dissertation&res_dat=xri:pqm&rft_dat=xri:pqdiss:10126520},
}

@book {Douglas1989Hilbert,
    AUTHOR = {Douglas, Ronald G. and Paulsen, Vern I.},
     TITLE = {Hilbert modules over function algebras},
    SERIES = {Pitman Research Notes in Mathematics Series},
    VOLUME = {217},
 PUBLISHER = {Longman Scientific \& Technical, Harlow; copublished in the
              United States with John Wiley \& Sons, Inc., New York},
      YEAR = {1989},
     PAGES = {vi+130},
      ISBN = {0-582-04796-X},
   MRCLASS = {46M20 (46H25 46J10 47A67 47B20 47D99)},
  MRNUMBER = {1028546},
MRREVIEWER = {Joseph\ A.\ Ball},
}

@article {Raeburn2003Countably,
    AUTHOR = {Raeburn, Iain and Thompson, Shaun J.},
     TITLE = {Countably generated {H}ilbert modules, the {K}asparov
              stabilisation theorem, and frames with {H}ilbert modules},
   JOURNAL = {Proc. Amer. Math. Soc.},
  FJOURNAL = {Proceedings of the American Mathematical Society},
    VOLUME = {131},
      YEAR = {2003},
    NUMBER = {5},
     PAGES = {1557--1564},
      ISSN = {0002-9939,1088-6826},
   MRCLASS = {46L08 (46L80)},
  MRNUMBER = {1949886},
MRREVIEWER = {Vladimir\ Manuilov},
       DOI = {10.1090/S0002-9939-02-06787-4},
       URL = {https://doi.org/10.1090/S0002-9939-02-06787-4},
}

@article {Li2010Hilbert,
    AUTHOR = {Li, Hanfeng},
     TITLE = {A {H}ilbert {$C^*$}-module admitting no frames},
   JOURNAL = {Bull. Lond. Math. Soc.},
  FJOURNAL = {Bulletin of the London Mathematical Society},
    VOLUME = {42},
      YEAR = {2010},
    NUMBER = {3},
     PAGES = {388--394},
      ISSN = {0024-6093,1469-2120},
   MRCLASS = {46L08 (42C15)},
  MRNUMBER = {2651932},
MRREVIEWER = {Vladimir\ Manuilov},
       DOI = {10.1112/blms/bdp109},
       URL = {https://doi.org/10.1112/blms/bdp109},
}

@article {Arambasic2017Frames,
    AUTHOR = {Aramba\v si\'c, Lj. and Baki\'c, D.},
     TITLE = {Frames and outer frames for {H}ilbert {$C^\ast$}-modules},
   JOURNAL = {Linear Multilinear Algebra},
  FJOURNAL = {Linear and Multilinear Algebra},
    VOLUME = {65},
      YEAR = {2017},
    NUMBER = {2},
     PAGES = {381--431},
      ISSN = {0308-1087,1563-5139},
   MRCLASS = {42C15 (46L05 46L08)},
  MRNUMBER = {3577457},
MRREVIEWER = {R.\ A.\ Zalik},
       DOI = {10.1080/03081087.2016.1186588},
       URL = {https://doi.org/10.1080/03081087.2016.1186588},
}

@book {Natterer1986Mathematics,
    AUTHOR = {Natterer, F.},
     TITLE = {The mathematics of computerized tomography},
 PUBLISHER = {B. G. Teubner, Stuttgart; John Wiley \& Sons, Ltd.,
              Chichester},
      YEAR = {1986},
     PAGES = {x+222},
      ISBN = {3-519-02103-X},
   MRCLASS = {44A15 (92A07)},
  MRNUMBER = {856916},
MRREVIEWER = {L.\ A.\ Shepp},
}

@misc{Landsman1998Lecture,
      title={Lecture notes on C*-algebras, Hilbert C*-modules, and quantum mechanics}, 
      author={N. P. Landsman},
      year={1998},
      eprint={math-ph/9807030},
      archivePrefix={arXiv},
      primaryClass={math-ph},
      url={https://arxiv.org/abs/math-ph/9807030}, 
}

\end{document}